\documentclass{amsart}

\usepackage{amsmath}
\usepackage{amsthm}
\usepackage{amsfonts}
\usepackage{amssymb}

\newcommand\R{{\mathbb{R}}}
\newcommand\C{{\mathbb{C}}}
\newcommand\Z{{\mathbf{Z}}}

\renewcommand\P{{\mathbf{P}}}
\newcommand\E{{\mathbf{E}}}

\renewcommand\Im{{\operatorname{Im}}}
\renewcommand\Re{{\operatorname{Re}}}

\newcommand\trace{\operatorname{trace}}
\newcommand\tr{\operatorname{trace}}

\newcommand\sgn{\operatorname{sgn}}

\newcommand\condone{{{\bf C1}}}

\newcommand\CE{{\mathcal E}}

\newcommand\ep{{\epsilon}}

\subjclass{15A52}

\theoremstyle{plain}
  \newtheorem{theorem}{Theorem}

  \newtheorem{proposition}[theorem]{Proposition}
  
  \newtheorem{lemma}[theorem]{Lemma}
  \newtheorem{corollary}[theorem]{Corollary}

\theoremstyle{definition}
  \newtheorem{definition}[theorem]{Definition}
  \newtheorem{example}[theorem]{Example}
  \newtheorem{remark}[theorem]{Remark}

\include{psfig}

\begin{document}

\title{A central limit theorem for the determinant of a Wigner matrix}

\author{Terence Tao}
\address{Department of Mathematics, UCLA, Los Angeles CA 90095-1555}
\email{tao@math.ucla.edu}
\thanks{T. Tao is supported by a grant from the MacArthur Foundation, by NSF grant DMS-0649473, and by the NSF Waterman award.}

\author{Van Vu}
\address{Department of Mathematics, Yale Univ., New Haven, CT 06520}
\email{van.vu@yale.edu}
\thanks{V. Vu is supported by research grants DMS-0901216 and AFOSAR-FA-9550-09-1-0167.}

\begin{abstract}  We establish a central limit theorem for the log-determinant 
$\log|\det(M_n)|$ of a Wigner matrix $M_n$, under the assumption of four matching moments with either the GUE or GOE ensemble.  More specifically, we show that this log-determinant is asymptotically distributed like $N( \log \sqrt{n!} - \frac{1}{2} \log n, \frac{1}{2} \log n )_\R$ when one matches moments with GUE, and $N( \log \sqrt{n!} - \frac{1}{4} \log n, \frac{1}{4} \log n )_\R$ when one matches moments with GOE.
\end{abstract}

\maketitle

\setcounter{tocdepth}{2}
\tableofcontents

\section{Introduction}

Random matrix theory is an important subject  in mathematics with applications to various 
areas such as numerical analysis, mathematical physics, statistics, number theory  and computer science, to mention a few. One of the main goals of this theory, by and large,  is to understand the distribution of various interesting functionals of a random matrix that naturally arise from linear algebra.

One of most natural and important matrix functionals is the \emph{determinant}.  As such, the study of determinants of random matrices has a long and rich history. The earlier papers on this study focused on the determinant $\det A_n$ of the non-Hermitian iid model $A_n$, where the entries $\zeta_{ij}$ of the matrix were independent random variables with mean $0$ and variance $1$. The earliest paper we find here
 belongs to  Szekeres and Tur\'an \cite{SzT}, in which they studied an extremal problem. In the 1950s, there were
a series of papers \cite{FT,NRR,Turan,Pre} devoted to the computation of  moments  of fixed orders of 
of the determinant  (see also \cite{Gbook}).  The explicit formula for higher moments get very complicated and in general not available, except in 
cases when the atom variables have some special distribution (see, for instance \cite{Dembo}).  

One can use the estimate for the moments and the Chebyshev inequality to obtain an upper bound on 
the magnitude $|\det A_n|$ of the determinant. However, no lower bound was known for a long time. In particular, Erd\H{o}s asked  whether $\det A_n$ is non-zero with probability tending to one.  In 1967, Koml\'os  \cite{Kom, Kom1}  addressed this question, proving  that  almost surely $|\det A_n |  > 0$ for random Bernoulli matrices (where the atom variables are iid Bernoulli, taking values $\pm 1$ with probability $1/2$). His method also works for 
much more general models.  Following \cite{Kom}, the upper  bound on the probability that 
$\det A_n =0$ has been improved in \cite{KKS, TVdet, TVsing, BVW}. However, these results do not say  much about the value of $|\det A_n |$ itself. 

A few years ago, the authors \cite{TVdet} managed to prove that for Bernoulli random matrices, 
with probability tending to one (as $n$ tends to infinity)
\begin{equation} \label{TVlow}  \sqrt {n !}  \exp( -c  \sqrt { n \log n} ) \le |\det A_n | \le  \sqrt {n!} \omega(n)
\end{equation} for any function $\omega(n)$ tending to infinity with $n$. This shows that almost surely,  $\log |\det A_n |$ is $ (\frac{1}{2} +o(1) )n \log n$, but does not otherwise provide much information on the limiting distribution of the log determinant. For related works concerning other models of random matrices, we refer to \cite{Ro}. 

In \cite{Goodman}, Goodman considered random Gaussian matrices $A_n = (\zeta_{ij})_{1 \leq i,j \leq n}$ where the atom variables $\zeta_{ij}$ are iid standard real Gaussian variables, $\zeta_{ij} \equiv N(0,1)_\R$. He noticed that in this case the square of the determinant can be expressed as the product of independent 
chi-square variables. Therefore, its logarithm is the sum of independent variables and thus one expects
a central limit theorem to hold. In fact, using properties of the chi-square distribution,  it is not hard to prove\footnote{Here and in the sequel, $\rightarrow$ denotes convergence in distribution.}
\begin{equation} \label{Girko} \frac{\log (|\det A_n|)- \frac{1}{2} \log n! + \frac{1}{2} \log n}{\sqrt{\frac{1}{2} \log n}}  \rightarrow N(0,1)_\R, \end{equation}
where $N(0,1)_\R$ denotes the law of the real Gaussian with mean $0$ and variance $1$; see e.g. \cite{rempala} for a proof.  Informally, we may write this law as
\begin{equation}\label{detain}
 |\det A_n| \approx n^{-1/2} \sqrt{n!} \exp( N(0,\frac{1}{2} \log n)_\R ).
\end{equation}
We remark that because the second moment of $\exp(N(0,t)_\R)$ is $e^{2t}$ for any $t>0$, this law is consistent with the second moment identity
\begin{equation}\label{turan}
 \E |\det A_n|^2 = n!,
\end{equation}
for iid matrices (and in particular, for Gaussian matrices) that was first observed by Tur\'an \cite{Turan}, and easily derivable from the Leibniz expansion
\begin{equation}\label{leibniz}
 \det A_n = \sum_{\sigma \in S_n} \sgn(\sigma) \prod_{i=1}^n \zeta_{i \sigma(i)}
\end{equation}
after observing that the terms on the right-hand side are pairwise uncorrelated in the iid case.
 
A similar analysis (but with the real chi distribution replaced by a complex chi distribution) also works for complex Gaussian matrices, in which $\zeta_{ij}$ remain jointly independent but now have the distribution of the complex Gaussian $N(0,1)_\C$ (or equivalently, the real and imaginary parts of $\zeta_{ij}$ are independent and have the distribution of $N(0,\frac{1}{2})_\R$).  In that case, one has a slightly different law
\begin{equation}\label{Girko-2} \frac{\log (|\det A_n|)- \frac{1}{2} \log n! + \frac{1}{4} \log n}{\sqrt{\frac{1}{4} \log n}}  \rightarrow N(0,1)_\R, \end{equation}
or more informally
\begin{equation}\label{detain-2}
 |\det A_n| \approx n^{-1/4} \sqrt{n!} \exp( N(0,\frac{1}{4} \log n)_\R ).
\end{equation}
Again, this remains consistent with \eqref{turan}.
 
We turn now to more general real iid matrices, in which the $\zeta_{ij}$ are jointly independent and real with mean zero and variance one.
In \cite{G2},  Girko stated that \eqref{Girko} holds 
for such random matrices  under the additional assumption
that the fourth moment of the atom variables is $3$.  Twenty years later,  he claimed a  
much stronger result which replaced the above assumption by the assumption that the atom variables have bounded $(4 +\delta)$-th moment \cite{G}. However, there are several points which are  not clear in 
these papers. Recently, Nguyen and the second author \cite{NVdet} gave a  new proof for  \eqref{Girko} under an exponential decay hypothesis on the entries.  Their approach also results in an estimate for the rate of convergence and is easily extended to handle to complex case. 

The analysis of the above random determinants relies crucially on the fact that the rows of the matrix are jointly independent.  This independence no longer holds for Hermitian random matrix models, which makes the analysis of determinants of Hermitian random matrices more challenging. The Hermitian version of Komlos' result \cite{Kom,Kom1} was 
posed as an open question by Weiss in the 1980s and was solved only five years ago \cite{CTV}
and for this purpose the authors needed to introduce the quadratic 
analogue of Littlewood-Offord-Erd\H os theorem. The analogue of \eqref{TVlow} was 
first proved in \cite[Theorem 31]{TVlocal1}, as a corollary\footnote{This theorem requires the atom variable has vanishing third moment, but one can remove this  requirement  using  very recent estimates of Nguyen \cite{nguyen} and Vershynin \cite{vershynin} on the least singular value.}  of the Four Moment theorem. But much as in the situation in the non-Hermitian case, these proofs do not reveal much information about the limiting distribution of the determinant. 

Let us now narrow down our consideration to the following class of random matrices.   
  
\begin{definition}[Wigner matrices]\label{def:Wignermatrix}  Let $n \geq 1$ be an integer. An $n \times n$ \emph{Wigner Hermitian matrix} $M_n$ is defined to be a  random Hermitian $n \times n$ matrix $M_n$ with upper triangular complex entries $\zeta_{ij}$ and diagonal real entries $\zeta_{ii}$ ($1\le i \le n$) jointly independent, with mean zero and variance one for $1 \leq i < j\leq n$, and mean zero and variance $\sigma^2$ for $1 \leq i \leq n$ and some $\sigma^2 > 0 $ independent of $n$.  We refer to the distributions of the $\zeta_{ij}$ as the \emph{atom distributions} of $M_n$.
\end{definition}

We say that the Wigner matrix ensemble \emph{obeys Condition {\condone}} for some constant $C_0$ if one has
\begin{equation*}
\E |\zeta_{ij}|^{C_0} \leq C_1
\end{equation*}
for all $1 \leq i,j \leq n$ and some constant $C_1$ independent of $n$.

\begin{example} The famous \emph{Gaussian Unitary Ensemble} (GUE) is the special case of the Wigner ensemble in which the atom distributions $\zeta_{ij}$ are given by the complex Gaussian $N(0,1)_\C$ for $1 \leq i < j \leq n$ and the real Gaussian $N(0,1)_\R$ for $1 \leq i=j \leq n$, thus in this case $\sigma^2 = 1$.  At the opposite extreme, the \emph{complex Hermitian Bernoulli ensemble} is an example of a discrete Wigner ensemble in which the atom distributions $\zeta_{ij}$ is equal to $\pm \frac{1}{\sqrt{2}} \pm \frac{1}{\sqrt{2}} \sqrt{-1}$ (with independent and uniform Bernoulli signs) for $1 \leq i < j \leq n$, and equal to $\pm 1$ for $1 \leq i=j \leq n$ (so again $\sigma^2=1$).

Another important example is the \emph{Gaussian Orthogonal Ensemble} (GOE) in which the atom distributions $\zeta_{ij}$ are given by $N(0,1)_\R$ for $1 \leq i < j \leq n$ and $N(0,2)_\R$ for $1 \leq i = j \leq n$, thus $\sigma^2=2$ in this case.  Finally, the \emph{symmetric Bernoulli ensemble} is an example in which $\zeta_{ij} \equiv \pm 1$ for all $1 \leq i \leq j \leq n$, with $\sigma^2=1$.

All of the above examples obey Condition \condone\ for arbitrary $C_0$.
\end{example}

We now consider the distribution of the determinant for Wigner matrices.  We first make the  observation that the first and second moments of the determinant are slightly different in the Wigner case than in the iid case:

\begin{theorem}[First and second moment]\label{secthm}  Let $M_n$ be a Wigner matrix.
\begin{itemize}
\item (First moment) When $n$ is odd, then $\E \det M_n = 0$.  When $n$ is even, one has
$$ \E \det M_n = (-1)^{n/2} \frac{n!}{(n/2)! 2^{n/2}}.$$
In particular, by Stirling's formula one has
$$ \E \det M_n = \left(\frac{(-1)^{n/2} 2^{1/4}}{\pi^{1/4}} + o(1)\right) n^{-1/4} \sqrt{n!}.$$
\item (Second moment)  If $M_n$ is drawn from GOE, then has\footnote{See Section \ref{not} for the asymptotic notation we will use, including Vinogradov's notation $\ll$.}
$$ n^{3/2} n! \ll \E |\det M_n|^2 \ll n^{3/2} n!,$$
while if $M_n$ is instead drawn from GUE, then
$$ n^{1/2} n! \ll \E |\det M_n|^2 \ll n^{1/2} n!.$$
\end{itemize}
\end{theorem}

\begin{proof} See Appendix \ref{second}.  More precise asymptotics for these moments in the GUE case were established by Brezin and Hikami \cite{brezin} (see also \cite{forrester}, \cite{garoni}, \cite{krasovsky}), however we give an elementary and self-contained proof of the above results in the appendix.
\end{proof}

Even in the GUE case, it is highly non-trivial to 
prove an analogue of the central limit theorem \eqref{Girko-2}; this was first achieved in \cite{delannay} via a lengthy computation using the explicit formula for the joint distribution of the eigenvalues. Notice that  the observation 
of Goodman does not apply due to the dependence between the rows and so it is not even clear why a central limit theorem must hold for the log-determinant. 
   
While it does not seem to be possible to express the log-determinant of GUE or GOE as a sum of independent random variables, in this paper we present a way to approximate the log-determinant as a sum of weakly dependent terms, based on\footnote{We would like to thank R. Killip for suggesting 
the use of Trotter's form.}  analysing a tridiagonal form of both GUE and GOE due to Trotter \cite{trotter}. 
Using stochastic calculus and the martingale central limit theorem
(see Section \ref{section:GUE}), we give a new proof of the following result.

\begin{theorem} 
[Central limit theorem for log-determinant of GUE and GOE]\label{clt-gue}  Let $M_n$ be drawn from GUE.  Then
$$ \frac{\log |\det(M_n)| - \frac{1}{2}  \log {n!} + \frac{1}{4} \log n }{\sqrt{\frac{1}{2} \log n}} \rightarrow N(0,1)_\R. $$

Similarly, if $M_n$ is drawn from GOE rather than GUE, one has 
$$ \frac{\log |\det(M_n)| - \frac{1}{2}  \log {n!} + \frac{1}{4} \log n }{\sqrt{\log n}} \rightarrow N(0,1)_\R. $$
\end{theorem} 

Informally, this theorem asserts that
$$ |\det M_n| \approx n^{-1/4} \sqrt{n!} \exp( N(0,\frac{1}{2} \log n)_\R )$$
for GUE, and
$$ |\det M_n| \approx n^{-1/4} \sqrt{n!} \exp( N(0,\log n)_\R )$$
for GOE (compare with \eqref{detain}, \eqref{detain-2}).  Note also that these distributions are consistent with the moment computations in Theorem \ref{secthm}.

As mentioned previously, Theorem \ref{clt-gue} has also been proven (using the explicit joint density distribution of the GUE and GOE eigenvalues) by Delannay and Le Caer \cite{delannay}.  However our approach is quite different in nature and somewhat less computational, and may be of independent interest.

The next task is to extend  beyond the GUE or GOE case.  Our main tool for this is a four moment theorem for log-determinants of Wigner matrices, analogous to the four moment theorems for eigenvalues \cite{TVlocal1}, \cite{TVlocal2}, \cite{TVlocal3}, Green's functions \cite{EYY}, and eigenvectors \cite{TV-vector}, \cite{knowles}.  Let us say that two Wigner matrices $M_n = (\zeta_{ij})_{1 \leq i, j \leq n}$ and $M'_n = (\zeta'_{ij})_{1 \leq i,j \leq n}$ \emph{match to order $m$} off the diagonal and to order $k$ on the diagonal if one has
$$ \E (\Re \zeta_{ij})^a  (\Im \zeta_{ij})^b =  \E (\Re \zeta'_{ij})^a  (\Im \zeta'_{ij})^b $$
for all $1 \leq i \leq j \leq n$ and natural numbers $a, b \geq 0$ with $a+b \leq m$ for $i<j$ and $a+b \leq k$ for $i=j$.

\begin{theorem}[Four moment theorem for determinant]\label{four-moment}  Let $M_n, M'_n$ be Wigner matrices whose atom distributions have independent real and imaginary parts that match to fourth order off the diagonal and to second order on the diagonal, are bounded in magnitude by $n^{c_0}$ for some sufficiently small but fixed $c_0 > 0$, and are supported on at least three points.  
Let $G: \R \to \R$ obey the derivative estimates
\begin{equation}\label{g-bound}
 |\frac{d^j}{dx^j} G(x)| =O( n^{c_0})
\end{equation}
for $0 \leq j \leq 5$.  Let $z_0 = E + \sqrt{-1}\eta_0$ be a complex number with $|E| \leq 2-\delta$ for some fixed $\delta>0$.  Then
$$ \E G( \log|\det (M_n-\sqrt{n} z_0)| ) - \E G( \log|\det(M'_n-\sqrt{n} z_0)| ) = O( n^{-c} )$$
for some fixed $c>0$, adopting the convention that $G(-\infty)=0$.

If $E=0$, then the requirement that the real and imaginary parts of the atom distribution are supported on at least three points can be dropped.
\end{theorem}

We prove this theorem in Section \ref{section:4moment}, following a preparation in Section 
\ref{section:swapping}.  The requirements that $M_n, M'_n$ be supported on at least three points, and that $E$ lie in the bulk region $|E| < 2-\delta$ are artificial, due to the state of current literature on level repulsion estimates (see Proposition \ref{lsv-lsl}).  It is likely that with further progress on those estimates that these hypotheses can be removed.  The hypothesis that the atom distributions have independent real and imaginary parts is mostly for notational convenience and can also be removed with some additional effort.  The hypothesis that the entries are bounded in magnitude by $n^{c_0}$ is, strictly speaking, not satisfied for distributions such as the Gaussian distribution, but in practice we will be able to reduce to this case by a truncation argument.

By combining Theorem \ref{four-moment} with Theorem \ref{clt-gue} we obtain

\begin{corollary}[Central limit theorem for log-determinant of Wigner matrices]\label{clt-wigner}  Let $M_n$ be a Wigner matrix whose atom distributions $\zeta_{ij}$ are independent of $n$, have real and imaginary parts that are independent and match GUE to fourth order, and obey Condition \condone for some sufficiently large absolute constant $C_0$.  Then
$$ \frac{\log |\det(M_n)| -  \frac{1}{2} \log n! + \frac{1}{4} \log n}{\sqrt{\frac{1}{2} \log n}} \rightarrow N(0,1)_\R. $$
If $M_n$ matches GOE instead of GUE, then one instead has
$$ \frac{\log |\det(M_n)| -  \frac{1}{2} \log n! + \frac{1}{4} \log n}{\sqrt{\log n}} \rightarrow N(0,1)_\R. $$
\end{corollary}

The deduction of this corollary from Theorem \ref{four-moment} and Theorem \ref{clt-gue} is standard (closely analogous, for instance, to the proof of \cite[Corollary 21]{TVlocal1}, which establishes a similar central limit theorem for individual eigenvalues of a Wigner matrix) and is omitted.
(Notice that in order for the atom variables of $M_n$ match those of GUE 
to fourth order, these variables must have at least three points in their supports.) 



\subsection{Notation}\label{not}  Throughout this paper, $n$ is a natural number parameter going off to infinity; in particular we will assume that $n \geq 100$ (so that $\log\log\log n$ is well-defined).  A quantity is said to be \emph{fixed} if it does not depend on $n$.  We write $X = O(Y)$, $X \ll Y$, or $Y \gg X$ if one has $|X| \leq CY$ for some fixed $C$, and $X = o(Y)$ if one has $X/Y \to 0$ as $n \to \infty$.  Absolute constants such as $C_0$ or $c_0$ are always understood to be fixed.

We say that an event $E$ occurs with \emph{high probability} if it occurs with probability $1-O(n^{-c})$ for some fixed $c > 0$, and with \emph{overwhelming probability} if it occurs with probability $1-O(n^{-A})$ for all fixed $A>0$.

\subsection{Acknowledgments} We thank Brad Rodgers and Zhigang Bao for references, Peter Eichelsbacher, Xiuyuan Cheng, and the anonymous referee for corrections, and Rowan Killip for suggesting and explaining the tridiagonal method.

\section{Central limit theorem for GUE} \label{section:GUE}

We now prove Theorem \ref{clt-gue}.  For notational reasons we shall take $n$ to be even, but the argument below can easily be verified to also work with minor modifications when $n$ is odd.  We will use a method suggested to us by Rowan Killip (private communication), and loosely based on the arguments in \cite{killip}.

We will work for most of this section with the GUE case, and discuss the changes in the numerology needed to address the GOE case at the end of the section.

The starting point is the following beautiful observation of Trotter \cite{trotter}:

\begin{proposition}[Tridiagonal form of GUE]\label{trotter}\cite{trotter} Let $M'_n$ be the random tridiagonal real symmetric matrix
$$ M'_n = \begin{pmatrix}
a_1 & b_1 & 0 & \ldots & 0 & 0 \\
b_1 & a_2 & b_2 & \ldots & 0 & 0 \\
0 & b_2 & a_3 & \ldots & 0 & 0 \\
\vdots & \vdots & \vdots & \ddots & \vdots & \vdots \\
0 & 0 & 0 &  \ldots & a_{n-1} & b_{n-1} \\
0 & 0 & 0 & \ldots & b_{n-1} & a_n
\end{pmatrix}$$
where the $a_1,\ldots,a_n, b_1,\ldots,b_{n-1}$ are jointly independent real random variables, with $a_1,\ldots,a_n \equiv N(0,1)_\R$ being standard real Gaussians, and each $b_i$ having a complex $\chi$-distribution:
$$ b_i = (\sum_{j=1}^i |z_{i,j}|^2)^{1/2}$$
where $z_{i,j} \equiv N(0,1)_\C$ are iid complex Gaussians\footnote{In other words, the real and imaginary parts of $z_{i,j}$ are independent with distribution $N(0,1/2)_\R$.}.  Let $M_n$ be drawn from GUE.  Then the joint eigenvalue distribution of $M_n$ is identical to the joint eigenvalue distribution of $M'_n$.
\end{proposition}

\begin{proof}  Let $M_n$ be drawn from GUE.  We can write
$$
M_n = \begin{pmatrix} 
M_{n-1} & X_n \\ X_n^* & a_n
\end{pmatrix}
$$
where $M_{n-1}$ is drawn from the $n-1\times n-1$ GUE, $a_n \equiv N(0,1)_\R$, and $X_n \in \C^{n-1}$ is a random Gaussian vector with all entries iid with distribution $N(0,1)_\C$.  Furthermore, $M_{n-1}, X_n, a_n$ are jointly independent.  

We now apply the tridiagonal matrix algorithm.
Let $b_{n-1} := |X_n|$, then $b_n$ has the $\chi$-distribution indicated in the proposition.  We then conjugate $M_n$ by a unitary matrix $U$ that preserves the final basis vector $e_n$, and maps $X_n$ to $b_{n-1} e_{n-1}$.  Then we have
$$
U M_n U^* = \begin{pmatrix} 
\tilde M_{n-1} & b_{n-1} e_{n-1} \\ b_{n-1} e_{n-1}^* & a_n
\end{pmatrix}
$$
where $\tilde M_{n-1}$ is conjugate to $M_{n-1}$.  Now we make the crucial observation: because $M_{n-1}$ is distributed according to GUE (which is a unitarily invariant ensemble), and $U$ is a unitary matrix independent of $M_{n-1}$, $\tilde M_{n-1}$ is also distributed according to GUE, and remains independent of both $b_{n-1}$ and $a_n$.  

We continue this process, expanding $U M_n U^*$ as
$$\begin{pmatrix} 
M_{n-2} & X_{n-1} & 0 \\
X_{n-1}^* & a_{n-1} & b_{n-1} \\
0 & b_{n-1} & a_n.
\end{pmatrix}
$$
Applying a further unitary conjugation that fixes $e_{n-1}, e_n$ but maps $X_{n-1}$ to $b_{n-2} e_{n-2}$, we may replace $X_{n-1}$ by $b_{n-2} e_{n-2}$ while transforming $M_{n-2}$ to another GUE matrix $\tilde M_{n-2}$ independent of $a_n, b_{n-1}, a_{n-1}, b_{n-2}$.  Iterating this process, we eventually obtain a coupling of $M_n$ to $M'_n$ by unitary conjugations, and the claim follows.
\end{proof}

In what follows, we are going to prove the limit law for the model $M'_n$ and hence for $M_n$. 
Since $b_i^2$ has expectation $i$  and variance\footnote{Note that the more familiar \emph{real} chi squared distribution $\chi_i^2$ would have variance $2i$ here, but $b_i^2$ has the \emph{complex} chi squared distribution which has variance $i$.} $i$,  we can write it as 
$$ b_i^2 = i + \sqrt {i} c_i, $$ 
where $c_i$ has mean $0$ and variance $1$. 

By the properties of normal distribution and chi square distribution (or from concentration of measure inequalities), we have the following tail bound. There are constants $C_1 , C_2 > 0 $ such that for all  $i \geq 1$ and $t \geq 0$, one has
\begin{equation} \label{tailbound} 
\P (\max\{ |a_i|, |c_i| \} \ge t ) \le C_1 \exp(-t^{C_2} ). 
\end{equation} 

Let $M'_i$ denote the upper left $i \times i$ minor of $M'_n$, and write $D_i := \det M'_i$.  From cofactor expansion we have the recursion
$$ D_i = a_i D_{i-1} - b_{i-1}^2 D_{i-2}$$
for all $i \geq 2$.  To prove Theorem \ref{clt-gue}, we need to establish the law
$$ \frac{\log |D_n| - \frac{1}{2}  \log {(n-1)!} -\frac{1}{4} \log n }{\sqrt{\frac{1}{2} \log n}} \rightarrow N(0,1)_\R. $$

It will be convenient to skip the first few terms of this recursion.  Let $m$ be a sufficiently slowly growing integer-valued function of $n$ (e.g. $m := \lfloor \log\log\log n \rfloor$ will suffice); we will only apply this recursion for $i \geq m$.   Notice that $D_i \neq 0$ with probability one for all $i$.

We then have
$$ D_i = a_i D_{i-1} - (i-1 + \sqrt{i-1} c_{i-1}) D_{i-2}.$$
To mostly eliminate the $i-1$ factor, we introduce the normalised determinants
$$ E_i := \frac{D_i}{\sqrt{i!}}$$
and conclude the recurrence
$$ E_i = \frac{a_i}{\sqrt{i}} E_{i-1} - \left( \frac{\sqrt{i-1}}{\sqrt{i}} + \frac{c_{i-1}}{\sqrt{i}}\right) E_{i-2}.$$
By Taylor expansion we can rewrite this as
\begin{equation}\label{est}
 E_i = \frac{a_i}{\sqrt{i}} E_{i-1} - \left( 1 + \frac{c_{i-1}}{\sqrt{i}} - \frac{1}{2i} + O( \frac{1}{i^2} )\right) E_{i-2}.
\end{equation}
Our task is now to show that
\begin{equation}\label{en-law}
 \frac{\log |E_n| + \frac{1}{4} \log n }{\sqrt{\frac{1}{2} \log n}} \rightarrow N(0,1)_\R. 
 \end{equation}

To deduce this central limit theorem from \eqref{est}, we would like to write $\log |E_i |$  as a sum of 
martingale differences. But it is rather hard to do from the above recursive formula \eqref{est}. 
We will need to perform an additional algebraic manipulation to obtain a more 
tractable formula involving the closely related quantity $F_j := E_{2j}^2 +E_{2j-1}^2$.  In particular, we will establish 

\begin{proposition}[Central limit theorem for $F_{n/2}$]\label{fn-clt}  We have
\begin{equation} \label{CLT1} \frac{\log F_{n/2} +  \frac{1}{2} \log n} {\sqrt {2 \log n}} \rightarrow N(0,1)_{\R}.\end{equation}
\end{proposition}

We now prove this proposition.  The idea is use Taylor expansions (which can be viewed as a discrete version of Ito's stochastic calculus) to approximate $\log F_{n/2} + \frac{1}{2} \log n$ as the sum $\sum_{j=1}^{n/2} \frac{1}{\sqrt{j}} h_j$ of martingale differences, to which the martingale central limit theorem may be applied.

We turn to the details.  From \eqref{est} for $i=2j,2j-1$ we first observe the crude bound
\begin{equation}\label{crude}
 F_j = O( Y_j^{O(1)} F_{j-1} ),
 \end{equation}
where $Y_j := 1 + |a_{2j}| + |c_{2j-1}| + |a_{2j-1}| + |c_{2j-2}|$.  Observe from \eqref{tailbound} that
\begin{equation}\label{xj-moment}
\E Y_j^l \ll 1
\end{equation}
for any fixed $l$.

Next, we apply \eqref{est} for $i=2j, 2j-1$ and use Taylor expansion (using \eqref{crude} to bound error terms of order $j^{-3/2}$ or better) to obtain
\begin{align*}
 E_{2j} &= \frac{a_{2j}}{\sqrt{2j}} E_{2j-1} - ( 1 + \frac{c_{2j-1}}{\sqrt{2j}} - \frac{1}{4j} ) E_{2j-2} + O\left( \frac{Y_j^{O(1)}}{j^{3/2}} F_{j-1}^{1/2} \right)\\
 E_{2j-1} &= (\frac{a_{2j-1}}{\sqrt{2j}} + \frac{r^{[1]}_j}{j}) E_{2j-2} - ( 1 + \frac{c_{2j-2}}{\sqrt{2j}} - \frac{1}{4j} + \frac{r_j^{[2]}}{j} ) E_{2j-3} + O\left( \frac{Y_j^{O(1)}}{j^{3/2}} F_{j-1}^{1/2} \right),
\end{align*}
where $r_j^{[1]}, r_j^{[2]}$ are random variables bounded in magnitude by $O(Y_j^{O(1)})$ and with mean zero.  (In fact, we can obtain a denominator of $j^2$ instead of $j^{3/2}$ in the error terms here, although this improved error term will not persist in our later analysis.)  Substituting the second equation into the first (and again using \eqref{crude} to handle all terms of order $j^{-3/2}$ or better), we also obtain
$$ E_{2j} = -( 1 + \frac{c_{2j-1}}{\sqrt{2j}} - \frac{1}{4j} + \frac{r_j^{[3]}}{j}) E_{2j-2} - (\frac{a_{2j}}{\sqrt{2j}} + \frac{r^{[4]}_j}{j}) E_{2j-3} +
O\left( \frac{Y_j^{O(1)}}{j^{3/2}} F_{j-1}^{1/2} \right)$$
where $r^{[3]}_j, r^{[4]}_j$ obey the same sort of bounds as $r_j^{[1]}, r_j^{[2]}$.  We may rewrite these estimates in matrix form as
\begin{equation}\label{ejp}
\begin{pmatrix} E_{2j} \\ E_{2j-1} \end{pmatrix}
= ( - 1
+ 
\frac{1}{\sqrt{2j}} G_j + \frac{1}{4j}  + \frac{1}{j} R_j )
\begin{pmatrix} E_{2j-2} \\ E_{2j-3} \end{pmatrix} + O\left( \frac{Y_j^{O(1)}}{j^{3/2}} F_{j-1}^{1/2} \right)
\end{equation}
where $G_j$ is the near-Gaussian matrix
\begin{equation}\label{gj-def}
G_j :=
\begin{pmatrix} 
- c_{2j-1} & -a_{2j} \\
a_{2j-1} & -c_{2j-2}
\end{pmatrix},
\end{equation}
and $R_j$ is a random matrix depending on $j, a_{2j-1}, a_{2j}, c_{2j-2}, c_{2j-1}$ with mean zero and whose entries are bounded by $O(|Y_j|^{O(1)})$.  (We remind the reader at this point that the implied constants in the $O()$ notation are independent of $j$.)

Using \eqref{ejp}, we can express
$$ F_j =  \begin{pmatrix} E_{2j} & E_{2j-1} \end{pmatrix} \begin{pmatrix} E_{2j} \\ E_{2j-1} \end{pmatrix}$$
as
$$ \begin{pmatrix} E_{2j-2} & E_{2j-3} \end{pmatrix} 
( - 1 + \frac{1}{\sqrt{2j}} G_j + \frac{1}{4j}  + \frac{1}{j} R_j )^*
( - 1 + \frac{1}{\sqrt{2j}} G_j + \frac{1}{4j}  + \frac{1}{j} R_j ) 
\begin{pmatrix} E_{2j-2} \\ E_{2j-3} \end{pmatrix} + O\left( \frac{Y_j^{O(1)}}{j^{3/2}} F_{j-1} \right).
$$
We can collect some terms, splitting $G_j^* G_j$ as the sum of $2$ and the mean zero random matrix $G_j^* G_j - 2$, and obtain the expansion\footnote{The $\frac{1}{2j}$ term here arises from combining three contributions $-1 \times \frac{1}{4j} + \frac{1}{4j} \times (-1) + \frac{1}{\sqrt{2j}} \times \frac{1}{\sqrt{2j}} \times 2$.}
\begin{equation}\label{multilog}
F_j = \left(1 + \frac{\sqrt{2}}{\sqrt{j}} h_j + \frac{1}{2j} + \frac{1}{j} k_j + O\left( \frac{Y_j^{O(1)}}{j^{3/2}} \right)\right) F_{j-1} 
\end{equation}
where
\begin{align*}
h_j&:= \frac{1}{F_{j-1}} 
\begin{pmatrix} E_{2j-2} & E_{2j-3} \end{pmatrix}
G_j \begin{pmatrix} E_{2j-2} \\ E_{2j-3} \end{pmatrix} \\
\frac{1}{F_{j-1}} 
\begin{pmatrix} E_{2j-2} & E_{2j-3} \end{pmatrix}
G_j^* \begin{pmatrix} E_{2j-2} \\ E_{2j-3} \end{pmatrix}
\end{align*}
and
$$ k_j := \frac{1}{F_{j-1}} \begin{pmatrix} E_{2j-2} & E_{2j-3} \end{pmatrix} R'_j \begin{pmatrix} E_{2j-2} \\ E_{2j-3} \end{pmatrix}$$
and $R'_j$ is a random matrix depending on $j, a_{2j-1}, a_{2j}, c_{2j-2}, c_{2j-1}$ with mean zero and entries bounded by $O(Y_j^{O(1)})$. 

We can expand $h_j$ as
\begin{equation} \label{hj} 
(-c_{2j-1}) \frac{E_{2j-2}^2} {E_{2j-2}^2 + E_{2j-3}^2 } +  (-c_{2j-2}) \frac{E_{2j-3}^2}{E_{2j-2}^2 + E_{2j-3}^2} 
+ (a_{2j-1} -a_{2j} ) \frac{E_{2j-2} E_{2j-3} } {E_{2j-2}^2 + E_{2j-3}^2} .\end{equation} 
As the $c_l, a_l$ are independent and all have mean $0$ and variance $1$, we conclude that for any fixed $E_{2j-2}$ and $E_{2j-3}$, $h_j$ also has mean zero and variance $1$, thus
\begin{equation} \label{hjcond} 
\E (h_j | \CE_{j-1} ) =0; \,\, \E (h_j^2| \CE_{j-1}) =1, 
\end{equation}  
where $\CE_l$ is the $\sigma$-algebra generated by the random variables $a_1,\ldots,a_{2l}$ and $c_1,\ldots,c_{2l-1}$ (or equivalently, by the entries of the minor $M_{2l}$).  
Similarly, for any fixed choice of $E_{2j-2}, E_{2j-3}$, $k_j$ is a real random variable with mean zero, and thus
\begin{equation}\label{joy}
\E (k_j | \CE_{j-1} ) =0.
\end{equation}
Also, from construction, $h_j, k_j = O(Y_j^{O(1)})$.  

Taking logarithms in \eqref{multilog}, we obtain
$$ \log F_j = \log F_{j-1} + \log \left(1+\frac{1}{2j} + \frac{\sqrt{2}}{\sqrt{j}} h_j +  \frac{1}{j} k_j + O\left( \frac{Y_j^{O(1)}}{j^{3/2}} \right) \right).$$
By telescoping series, we may thus write
$$ \log F_{n/2} = \log F_m + \sum_{j=m+1}^{n/2} \log\left(1 + x_j + O\left( \frac{Y_j^{O(1)}}{j^{3/2}} \right) \right)$$
where
\begin{equation}\label{xj}
 x_j := \frac{1}{2j} + \frac{\sqrt{2}}{\sqrt{j}} h_j +  \frac{1}{j} k_j.
 \end{equation}
For $m$ sufficiently slowly growing in $n$, we clearly have
$$ \log F_m = o( \sqrt{\log n} )$$
with probability $1-o(1)$, since $F_m$ is almost surely finite with a law that depends only on $m$ and not on $n$.  To prove \eqref{CLT1}, it thus suffices to show that
\begin{equation} \label{CLT2} \frac{\sum_{j=m+1}^{n/2} \log\left(1 + x_j + O\left( \frac{Y_j^{O(1)}}{j^{3/2}} \right) \right) +  \frac{1}{2} \log n} {\sqrt {2 \log n}} \rightarrow N(0,1)_{\R}.\end{equation}
The next step is to use Taylor expansion to approximate the logarithm to extract something that more closely resembles a martingale difference.  Observe that $x_j = O( Y_j^{O(1)} / j^{1/2} )$.  From \eqref{xj-moment}, we conclude that with probability $1-O(j^{-100})$ (say), the expression $1 + x_j + O( \frac{Y_j^{O(1)}}{j^{3/2}} )$ lies between $1/2$ and $3/2$ (say).  From the union bound, we thus see that with probability
$$ 1 - \sum_{j=m+1}^{n/2} O(j^{-100}) = 1-o(1),$$
one has
$$ \log\left(1 + x_j + O\left( \frac{Y_j^{O(1)}}{j^{3/2}} \right) \right) = x_j - x_j^2/2 + O\left( \frac{Y_j^{O(1)}}{j^{3/2}} \right)$$
for all $m+1 \leq j \leq n/2$.  As $h_j$ has variance $1$, we can split $h_j^2$ as the sum of $1$ and the mean zero random variable $h_j^2-1$. From
\eqref{xj} we may thus expand
$$ x_j - x_j^2/2 = -\frac{1}{2j} + \frac{\sqrt{2}}{\sqrt{j}} h_j + \frac{1}{j} k'_j + O\left( \frac{Y_j^{O(1)}}{j^{3/2}} \right)$$
where $k'_j$ is a random variable bounded by $O(Y_j^{O(1)})$ which has conditional mean zero:
$$\E (k'_j | \CE_{j-1} ) =0.$$

Similarly, from \eqref{xj-moment} and the union bound again, the $O( Y_j^{O(1)} / j^{3/2} )$ error terms are $O(1/j^{1.1})$ (say) with probability $1-o(1)$, and thus we see that with probability $1-o(1)$, we have
\begin{equation}\label{ass}
 \sum_{j=m+1}^{n/2} \log\left(1 + x_j + \left( \frac{Y_j^{O(1)}}{j^{3/2}} \right) \right) = 
\sum_{j=m+1}^{n/2} \frac{\sqrt{2}}{\sqrt{j}} h_j + \frac{1}{j} k'_j  - \frac{1}{2} \log n + O(1).
\end{equation}
To prove \eqref{CLT2}, it thus suffices to show that
\begin{equation}\label{CLT3}
\frac{\sum_{j=m+1}^{n/2} \frac{\sqrt{2}}{\sqrt{j}} h_j + \frac{1}{j} k'_j}{\sqrt{2\log n}} \rightarrow N(0,1)_{\R}.
\end{equation}
Observe that as each $k'_j$ are martingale differences, which have variance $O(1)$ thanks to \eqref{xj-moment}.  As such, the expression $\sum_{j=m+1}^{n/2} \frac{1}{j} k'_j$ has variance $\sum_{j=m+1}^{n/2} O(1/j^2) = o(\log n)$ and can thus be discarded.  If $m$ is small enough, then the expression $\sum_{j=1}^m \frac{\sqrt{2}}{\sqrt{j}} h_j$ has variance $o(\log n)$ and can similarly be discarded.  Thus it suffices to show that
\begin{equation}\label{CLT4}
\frac{\sum_{j=1}^{n/2} \frac{1}{\sqrt{j}} h_j}{\sqrt{\log n}} \rightarrow N(0,1)_{\R}.
\end{equation}
In order to verify \eqref{CLT4}, we need to invoke the martingale central limit theorem:

\begin{theorem}[Martingale central limit theorem]\label{theorem:martingale}
Assume that $T_{1},\dots,T_{n}$ are martingale differences with respect to the nested $\sigma$-algebra $\CE_0, \CE_{1},\dots, \CE_{n-1}$. Let $v_n^2 := \sum_{i=1}^{n}\E(T_{i}^2|\CE_{i-1})$, and $s_m^2:= \sum_{i=1}^{n}\E(T_{i}^2)$. Assume that 

\begin{itemize}
\item $v_n/s_n \rightarrow 1$ in probability;
\vskip .2in

\item  (Lindeberg condition) for every $\ep>0$, $s_n^{-2}\sum_{i=0}^{n-1}\E(T_{i+1}^2\mathbf{1}_{T_{i+1}\ge \ep s_n}) \rightarrow 0$ as $m\rightarrow \infty$. 
\end{itemize}

Then $\frac{\sum_{i=1}^n T_{i}}{s_n } \rightarrow N(0,1)_\R$.
\end{theorem}

\begin{proof} See \cite[Theorem 1]{Br}.
\end{proof}

We apply this theorem with $T_j := \frac{1}{\sqrt{j}} h_j$.  From \eqref{hjcond} one has
$\E(T_{j+1} |\CE_j )=0$ and $\E(T_{j+1}^2 |\CE_j )=\frac{1}{j}$, and hence also $\E(T_{j+1}^2) = \frac{1}{j}$.  Thus $v_n=s_n = \log^{1/2} n + O(1)$; this gives the first hypothesis in Theorem \ref{theorem:martingale}.

Now we verify the Lindeberg condition.  From \eqref{xj-moment} we have $\E(T_{i+1}^4|\CE_i) = \frac{1}{(i+1)^2} \E(h_{i+1}^4|\CE_i) \ll \frac{1}{i^2}$ and hence
$$ \E(T_{i+1}^2\mathbf{1}_{T_{i+1}\ge \ep s_n})  \leq \ep^{-2} s_n^{-2} \E T_{i+1}^4 \ll \frac{\ep^{-2}}{i^2 s_n};$$
since $s_n = \log^{1/2} n + O(1)$, the claim follows.  This concludes the proof of Proposition \ref{fn-clt}.

Proposition \ref{fn-clt} controls the magnitude of the vector $\begin{pmatrix} E_{2j} \\ E_{2j-1} \end{pmatrix}$ when $j=n/2$.  We will however be interested in the distribution of $E_n$, and so we must also obtain information about the \emph{phase} of this vector also.  To this end, we express this vector in polar coordinates as
\begin{equation}\label{polar}
\begin{pmatrix} E_{2j} \\ E_{2j-1} \end{pmatrix} = (-1)^j F_j^{1/2} \begin{pmatrix} \cos \theta_j \\ \sin \theta_j \end{pmatrix}
\end{equation}
for some $\theta_j \in \R/2\pi \Z$, where we introduce the sign $(-1)^j$ to cancel the $-1$ factor in \eqref{ejp}.  

\begin{proposition}[Uniform distribution of $\theta_n$]  One has
$$ \theta_n \rightarrow {\bf u}$$
as $n \to \infty$, where ${\bf u}$ is the uniform distribution on $\R/2\pi\Z$.
\end{proposition}

\begin{proof} By the Weyl equidistribution criterion, it suffices to show that
$$ \E e^{i k \theta_n} = o(1)$$
for every fixed non-zero integer $k$.

Fix $k$. Inserting the polar representation \eqref{polar} into \eqref{ejp}, we obtain the recursion
\begin{equation}\label{ejp-2}
\begin{pmatrix} \cos \theta_j \\ \sin \theta_j \end{pmatrix}
= C_j ( 1 + D_j )
\begin{pmatrix} \cos \theta_{j-1} \\ \sin \theta_{j-1} \end{pmatrix} 
\end{equation}
where $D_j$ is the matrix
\begin{equation}\label{rjf}
D_j := 
- \frac{1}{\sqrt{2j}} G_j + \frac{1}{j} R''_j + O( \frac{Y_j^{O(1)}}{j^{3/2}} ),
\end{equation}
where $R''_j$ is a matrix obeying the same properties as $R_j$ or $R'_j$, and $C_j$ is a non-zero scalar whose exact value is not important for us.  

We extract the components $e_j, f_j$ of $D_j$ in 
  in the orthonormal basis formed by  the two vectors
$(\cos \theta_{j-1}, \sin \theta_{j-1})$ and $(-\sin \theta_{j-1}, \cos
\theta_{j-1}) $, thus
$$ e_j := \begin{pmatrix} -\sin \theta_{j-1}, & \cos \theta_{j-1} \end{pmatrix} 
D_j \begin{pmatrix} \cos \theta_{j-1} \\ \sin \theta_{j-1} \end{pmatrix} $$
and
$$ f_j := \begin{pmatrix} \cos \theta_{j-1}, & \sin \theta_{j-1} \end{pmatrix} 
D_j \begin{pmatrix} \cos \theta_{j-1} \\ \sin \theta_{j-1} \end{pmatrix}.
$$

From \eqref{ejp-2} we thus have
$$\frac{1}{C_j} (\cos \theta_j, \sin \theta_j) = (1+f_j) (\cos \theta_{j-1},
\sin \theta_{j-1}) + e_j (-\sin \theta_{j-1}, \cos  \theta_{j-1}), $$
\noindent and so we have a right-angled triangle with base $ 1+f_j$,  height $e_j$,
and angle $\theta_j- \theta_{j-1}$.  Elementary trigonometry then gives 
$$ \tan(\theta_j-\theta_{j-1}) = \frac{e_j}{1+f_j}.$$

By \eqref{xj-moment}, we see that with probability $1-O(j^{-100})$ (say), we have $e_j, f_j = O(j^{-0.49})$ (say). Using the  Taylor expansion 
of $\operatorname{arc} \tan$  and $1/(1+x)$ we obtain 

$$ \theta_j - \theta_{j-1} = e_j - e_j f_j + O( j^{-1.47} )$$
and thus
$$ e^{ik\theta_j} = e^{ik\theta_{j-1}} ( 1 + ik e_j - ik e_j f_j - \frac{k^2}{2} e_j^2 + O( j^{-1.47} ) )$$
with probability $1-O(j^{-100})$.  Hence
$$ \E e^{ik\theta_j} = \E e^{ik\theta_{j-1}} ( 1 + ik e_j - ik e_j f_j - \frac{k^2}{2} e_j^2 ) + O( j^{-1.47} ).$$
Now from \eqref{rjf}, \eqref{gj-def}, \eqref{xj-moment} we see that after conditioning  on $\theta_{j-1}$,
 $e_j$ has mean $O(j^{-1.47})$ and variance $\frac{1}{2j} + O(j^{-1.47})$, and that $e_j f_j$ has mean $O(j^{-1.47})$.  We conclude that
$$\E e^{ik\theta_j} = \E e^{ik\theta_{j-1}} (1 - \frac{k^2}{4j}) + O(j^{-1.47})$$
for any $1 \leq j \leq n$.
Telescoping this, we see that
$$|\E e^{ik\theta_n}| \ll (\frac{m}{n})^{k^2/4} |\E e^{ik\theta_m}| + O( m^{-0.47})$$
for any $1 \leq m \leq n$.  Bounding $|\E e^{ik\theta_m}|$ by $1$ and choosing $m$ to be a slowly growing function of $n$, we obtain the claim.
\end{proof}

From the above proposition, we see in particular that
$$\frac{1}{\log n} \leq |\cos \theta_{n/2}| \leq 1$$ 
(say) with probability $1-o(1)$.  Since $E_n = (-1)^{n/2} F_{n/2}^{1/2} \cos \theta_{n/2}$, we thus have
$$\log |E_n| = \frac{1}{2} \log F_{n/2} + O( \log \log n )$$ 
with probability $1-o(1)$; combining this with Proposition \ref{fn-clt}, we see that
$$ \frac{\log|E_n|^2 + \frac{1}{2} \log n}{\sqrt{2 \log n}} \to N(0,1)_\R.$$
The claim \eqref{en-law} then follows.

\subsection{The GOE case}

We now discuss the changes to the above argument needed to address the GOE case.  The analogue of Proposition \ref{trotter} is easily established, but with the changes that the $a_j$ now have the distribution of $N(0,2)_\R$ instead of $N(0,1)_\R$, and the $b_j$ now have a real $\chi$-distribution instead of a complex one (thus the $z_{i,j}$ are now distributed according to $N(0,1)_\R$ instead of $N(0,1)_\C$).  The effect of this is to make the random variables $a_j,c_j$ in the above analysis have variance $2$ instead of $1$ (but they still have mean zero).  As a consequence $G^* G$ now has mean $4$ rather than mean $2$, which means that the $\frac{1}{2j}$ term in \eqref{multilog} becomes $\frac{3}{2j}$.  On the other hand, the random variables $h_j$ now have variance $2$ instead of $1$.  These two changes cancel each other out to some extent, and in particular the assertion \eqref{ass} remains unchanged.  Finally, when applying the martingale central limit theorem, the variances $v_n^2, s_n^2$ are now $2 \log n + O(1)$ rather than $\log n+O(1)$, again thanks to the increased variance of $h_j$.  The remainder of the argument goes through with the obvious changes.

\section{Resolvent swapping: a deterministic analysis} \label{section:swapping} 

In this section we study the stability of Hermitian matrices with respect to perturbation in just one or two entries.  To formalise this we will need some definitions.

We will need a number of matrix norms.  Let $A = (a_{ij})_{1 \leq i,j \leq n}$ be a matrix, and let $1 \leq p,q \leq \infty$ be exponents.  We use $\|A\|_{(q,p)}$ to denote the $\ell^p \to \ell^q$ operator norm, i.e. the best constant in the inequality
$$ \|Ax\|_{\ell^q} \leq \|A\|_{(q,p)} \|x\|_{\ell^p}$$
Thus for instance $\|A\|_{(2,2)}$ is the usual operator norm.  We also observe the identities
$$\|A\|_{(\infty,1)} = \sup_{1 \leq i,j \leq n} |a_{ij}|$$
and
$$ \|A\|_{(\infty,2)} := \sup_{1 \leq i \leq n} (\sum_{j=1}^n |a_{ij}|^2)^{1/2}.$$
In particular one has the identity
\begin{equation}\label{tts}
\|A\|_{(\infty,2)} = \|A A^*\|_{(\infty,1)}^{1/2}.
\end{equation}
By duality one has
\begin{equation}\label{dual}
\|A\|_{(q,p)} = \|A^*\|_{(p',q')},
\end{equation} where $1/p +1/p'= 1/q +1/q'=1$.

We observe the trivial inequality
\begin{equation}\label{op-dec}
\| AB \|_{(r,p)} \leq \|A\|_{(r,q)} \|B\|_{(q,p)}
\end{equation}
for any $A,B$ and $1 \leq p,q,r \leq \infty$.

Next, we need the notion of an elementary matrix.

\begin{definition}[Elementary matrix]  An \emph{elementary matrix} is a matrix which has one of the following forms
\begin{equation}\label{vform}
 V = e_a e_a^*, e_a e_b^* + e_b e_a^*, \sqrt{-1} e_a e_b^* - \sqrt{-1} e_b e_a^*
\end{equation}
with $1 \leq a,b \leq n$ distinct, where $e_1,\ldots,e_n$ is the standard basis of $\C^n$.
\end{definition}

Observe that
\begin{equation}\label{voi}
 \| V\|_{(q,p)} \ll 1
\end{equation}
and
\begin{equation}\label{votr}
|\tr(AV)| = |\tr(VA)| =O( \|A\|_{(q,p)})
\end{equation}
for all $1 \leq p,q \leq \infty$ and all $n \times n$ matrices $A$.  

Let $M_0$ be a Hermitian matrix, let $z=E+i\eta$ be a complex number, and let $V$ be an elementary matrix.  We then introduce, for each $t \in \R$, the Hermitian matrices
$$ M_t := M_0 + \frac{1}{\sqrt{n}} tV,$$
the resolvent
\begin{equation}\label{resolve}
R_t = R_t(E+i\eta) := (M_t - E - i\eta)^{-1}
\end{equation}
and the Stieltjes transform
$$ s_t := s_t(E+i\eta) := \frac{1}{n} \tr R_t(E + i\eta)$$
and study how $R_t$ and $s_t$ depend on $t$.  

We have the fundamental \emph{resolvent identity}
$$ R_t = R_0 - \frac{t}{\sqrt{n}} R_0 V R_t$$
which upon iteration leads to 
\begin{equation}\label{rotkv}
 R_t = R_0 + \sum_{j=1}^k \left(-\frac{t}{\sqrt{n}}\right)^j (R_0 V)^j R_0 + \left(-\frac{t}{\sqrt{n}}\right)^{k+1} (R_0 V)^{k+1} R_t. 
\end{equation}

Under a mild hypothesis, we also have the infinite limit of \eqref{rotkv}:

\begin{lemma}[Neumann series]\label{neum}  Let $M_0$ be a Hermitian $n \times n$ matrix, let $E \in \R$, $\eta > 0$, and $t \in \R$, and let $V$ be an elementary matrix. Suppose one has
\begin{equation}\label{tro}
 |t| \| R_0\|_{(\infty,1)} = o(\sqrt{n}).
\end{equation}
Then one has the Neumann series formula
\begin{equation}\label{neumann}
 R_t = R_0 + \sum_{j=1}^\infty \left(-\frac{t}{\sqrt{n}}\right)^j (R_0 V)^j R_0 
\end{equation}
with the right-hand side being absolutely convergent, where $R_t$ is defined by \eqref{resolve}.  Furthermore, for any $1 \leq p \leq \infty$ one has
\begin{equation}\label{tp}
 \|R_t\|_{(\infty,p)} \leq (1+o(1)) \|R_0\|_{(\infty,p)}.
\end{equation}
\end{lemma}

In practice, we will have $t = n^{O(c_0)}$ (from a decay hypothesis on the atom distribution) and $\|R_0\|_{(\infty,1)} = n^{O(c_0)}$ (from eigenvector delocalisation and a level repulsion hypothesis), where $c_0>0$ is a small constant, so \eqref{tro} is quite a mild condition.  We also remark that by replacing $M_0$ and $t$ with $M_0+tV$ and $-t$ respectively, one can swap the roles of $R_0$ and $R_t$ in the above lemma without difficulty.

\begin{proof}   We rewrite \eqref{rotkv} as
\begin{equation}\label{rotkv-2}
\left(1 - \left(-\frac{t}{\sqrt{n}}\right)^{k+1} (R_0 V)^{k+1}\right) R_t = R_0 + \sum_{j=1}^k \left(-\frac{t}{\sqrt{n}}\right)^j (R_0 V)^j R_0 
\end{equation}
for all $k \geq 0$.  Sending $k \to \infty$ we will be able to conclude \eqref{neumann} (in a conditionally convergent sense, at least) once we show that $\left(-\frac{t}{\sqrt{n}}\right)^{k+1} (R_0 V)^{k+1}$ converges to zero (in, say, $\|A\|_{(\infty,1)}$ norm) as $k \to \infty$.  But from \eqref{voi}, \eqref{op-dec} we have
$$ \left\| \left(-\frac{t}{\sqrt{n}}\right)^{k+1} (R_0 V)^{k+1} \right\|_{(\infty,1)} \leq O\left(\frac{|t|}{\sqrt{n}}\right)^{k+1} \| R_0\|_\infty^{k+1}.$$
From \eqref{tro}, this decays exponentially in $k$, and this gives \eqref{neumann} (and also demonstrates that the series is absolutely convergent).

Taking $(\infty,p)$ norms of \eqref{neumann}, one has
$$
\|R_t\|_{(\infty,p)} \leq \|R_0\|_{(\infty,p)} + \sum_{j=1}^\infty \left(\frac{|t|}{\sqrt{n}}\right)^j \| (R_0 V)^j R_0 \|_{(\infty,p)}.$$
But from \eqref{voi}, \eqref{op-dec} one has
$$ \| (R_0 V)^j R_0 \|_{(\infty,p)} \leq (\| R_0\|_{\infty,1} \|V\| _{1 ,\infty}\|)^j \|R_0\| _{\infty, p} =  O(\|R_0\|_{\infty,1})^j  \|R_0\|_{(\infty,p)}$$
and the claim \eqref{tp} follows from \eqref{tro}.
\end{proof}

We now can describe the dependence of $s_t$ on $t$:

\begin{proposition}[Taylor expansion of $s_t$]\label{taylor-prop}  Let the notation be as above, and suppose that \eqref{tro} holds.  Let $k \geq 0$ be fixed.   Then one has
\begin{equation}\label{expand}
 s_t = s_0 + \sum_{j=1}^k n^{-j/2} c_j t^j + O\left( n^{-(k+1)/2} |t|^{k+1} \|R_0\|_{(\infty,1)}^{k+1} \min( \|R_0\|_{(\infty,1)}, \frac{1}{n\eta} ) \right)
 \end{equation}
where the coefficients $c_j$ are independent of $t$ and obey the bounds
\begin{equation}\label{cj-bound}
c_j  =O\left(\|R_0\|_{(\infty,1)}^{j} \min \{ \|R_0\|_{(\infty,1)}, \frac{1}{n\eta} \}\right) .
\end{equation}
for all $1 \leq j \leq k$.
\end{proposition}

\begin{proof} If we take normalised traces of \eqref{rotkv}, we obtain
$$ s_t = s_0 + \sum_{j=1}^k n^{-j/2} c_j t^j + n^{-(k+1)/2} t^{k+1} r^{(k)}_t$$
where $c_j$ are the coefficients
$$
c_j := (-1)^j \frac{1}{n} \tr( (R_0 V)^j R_0 )$$
and $r^{(k)}_t$ is the error term
$$ r^{(k)}_t := (-1)^{k+1} \frac{1}{n} \tr( (R_0 V)^{k+1} R_t ).$$

To estimate these coefficients, we use the cyclic property of trace to rearrange
$$ \tr( (R_0 V)^j R_0 ) = \trace (V (R_0 V)^{j-1} R_0^2)$$
and thus by \eqref{votr}, \eqref{op-dec}, \eqref{voi} we have
\begin{align*}
 |c_j| &\ll \frac{1}{n} \|(R_0 V)^{j-1} R_0^2\|_{(\infty,1)} \\
 &\ll \frac{1}{n} \| R_0 V \|_{(\infty,\infty)}^{j-1} \|R_0^2\|_{(\infty,1)} \\
 &\ll \frac{1}{n} ( \| R_0 \|_{(\infty,1)} \| V \|_{(1,\infty)})^{j-1} \|R_0^2\|_{(\infty,1)} \\
 &\ll \frac{1}{n} \|R_0\|_{(\infty,1)}^{j-1} \|R_0^2\|_{(\infty,1)}).
\end{align*}
We can bound this in one of two ways.  Firstly, by \eqref{op-dec} we have the crude inequality
$$ \|R_0^2\|_{(\infty,1)} \leq \|R_0\|_{(\infty,\infty)} \|R_0\|_{(\infty,1)} \leq
n \|R_0\|_{(\infty,1)}^2$$
(coming from the bound $\|x\|_{\ell^1} \leq n \|x\|_{\ell^\infty}$), leading to the bound
$$ c_j = O(\|R_0\|_{(\infty,1)}^{j+1}).$$
Alternatively,  we can use \eqref{op-dec}, \eqref{dual}, \eqref{tts} to bound
\begin{align*}
\|R_0^2\|_{(\infty,1)} &\leq \|R_0\|_{(\infty,2)} \|R_0\|_{(2,1)} \\
&\leq \| R_0\|_{(\infty,2)} \|R_0^*\|_{(\infty,2)} \\
&\leq \| R_0 R^*_0\|_{(\infty,1)}^{1/2} \|R_0^* R_0 \|_{(\infty,1)}^{1/2}.
\end{align*}
But from the definition \eqref{resolve} of the resolvent, one has the identity
$$ R_0 R^*_0 = R^*_0 R_0 = \frac{R_0 - R_0^*}{2i\eta}$$
and thus from the triangle inequality and \eqref{dual}
$$ \|R_0^2\|_{(\infty,1)} \leq \frac{1}{\eta} \|R_0\|_{(\infty,1)}.$$
This gives
$$ c_j =O\left( \frac{1}{n\eta} \|R_0\|_{(\infty,1)}^{j}\right).$$
Combining the two bounds on $c_j$ we have \eqref{cj-bound}.
A similar argument can be used to bound $r^{(k)}_t$ (using \eqref{tp} to replace $R_t$ by $R_0$ at some stage of the argument), so that
\begin{equation}\label{rkt-bound}
r^{(k)}_t =O\left(\|R_0\|_{(\infty,1)}^{k+1} \min\{ \|R_0\|_{(\infty,1)}, \frac{1}{n\eta} \}\right).
\end{equation}
The claim \eqref{expand} follows.
\end{proof}

\section{Proof of Theorem \ref{four-moment}} \label{section:4moment}

In this section we prove Theorem \ref{four-moment}.  Let $M_n, M'_n$ be as in that theorem, with $c_0$ sufficiently small to be chosen later.  Call a statistic $S(M)$ that can depend on a matrix $M$ \emph{highly insensitive} if one has 
$$ |S(M_n) - S(M'_n)| =O(n^{-c})$$
for some fixed $c>0$.  Thus our task is to show that $\E G( \log |\det (M_n-\sqrt{n}z_0)| )$ is highly insensitive for all $z_0$ and all $G$ obeying \eqref{g-bound}.  By dividing $G$ by $n^{c_0}$ (and reducing the size of $c_0$ if necessary) we may improve \eqref{g-bound} to the estimates
\begin{equation}\label{jx}
|\frac{d^j}{dx^j} G(x)| \leq 1
\end{equation}
for all $x \in \R$ and $0 \leq j \leq 5$.

By truncating the atom distributions (and re-adjusting to keep them at mean zero and unit variance) and using Condition \condone, we may assume without loss of generality that we have the uniform upper bound
\begin{equation}\label{xio}
 |\xi| \ll n^{c_0}
 \end{equation}
on the atom distribution (see \cite[Chapter 2]{BSbook} or \cite[Appendix A]{NVdet} for more details on the truncation technique).

Set $W_n := \frac{1}{\sqrt{n}} M_n$ (and $W'_n := \frac{1}{\sqrt{n}} M'_n$).  Then
$$ \log |\det (M_n-\sqrt{n}z_0)| = \frac{1}{2} n \log n + \log|\det(W_n-z_0)|.$$
By translating $G$ by $\frac{1}{2} n \log n$ (which does not affect the bounds \eqref{jx}), it thus suffices to show that $\E G(\log|\det(W_n-z_0)|)$ is highly insensitive.

Write $z_0 = E + \sqrt{-1} \eta_0$.  By conjugation symmetry we may take $\eta_0 \geq 0$.  We first  dispose of the easy case when $\eta_0 \geq n^{100}$.    In this case we have
$$  \log |\det(W_n - z_0)| = n \log |z_0| +  \log|\det(1 - z_0^{-1} W_n)| = n \log |z_0| + O(n^{-50})$$
(say), thanks to \eqref{xio}.  The claim then follows easily in this case from \eqref{jx}.

We now  restrict to the main case $0 \leq \eta_0 < n^{100}$.  From the fundamental theorem of calculus one has
$$ \log |\lambda-z_0| = \log |\lambda - E-\sqrt{-1} n^{100}| - \Im \int_{\eta_0}^{n^{100}} \frac{d\eta}{\lambda - E-\sqrt{-1}\eta}$$
and hence
\begin{equation}\label{helo}
  \log|\det(W_n-z_0)| =  \log |\det(W_n - E-\sqrt{-1} n^{100})| - n \Im \int_{\eta_0}^{n^{100}} s(E+\sqrt{-1}\eta)\ d\eta
\end{equation}
where 
$$s(z) = s_{W_n}(z) = \frac{1}{n} \tr(W_n-z)^{-1}$$
is the Stieltjes transform of $W_n$.  

\noindent The previous analysis and \eqref{helo} then gives
$$\log|\det(W_n-z_0)| = n \log n^{100} + O(n^{-50}) - n \Im \int_{\eta_0}^{n^{100}} s(E+\sqrt{-1}\eta)\ d\eta$$
By translating (and reflecting) $G$ once more, it thus suffices to show that the quantity
$$ \E G\left( n \Im \int_{\eta_0}^{n^{100}} s(E+\sqrt{-1}\eta)\ d\eta \right)$$
is highly insensitive.


We next need the following proposition.   Let $\lambda_1(W_n) \geq \ldots \geq \lambda_n(W_n)$ denote the eigenvalues of $W_n$ (counting multiplicity), and let $u_1(W_n),\ldots,u_n(W_n)$ be an associated orthonormal basis of eigenvectors.

\begin{proposition}[Non-concentration]\label{lsv-lsl} With high probability, one has
$$ \min_{1 \leq i \leq n} |\lambda_i(W_n)-E| \geq n^{-1-c_0}$$
and with overwhelming probability one has
$$ N_I =O( n |I|) $$
whenever $I$ is an interval of length $|I| \geq n^{-1+Ac_0}$ for a sufficiently large constant $A>0$.  Also, with overwhelming probability one has
$$ \sup_{1 \leq i \leq n} \|u_i(W_n)\|_{\ell^\infty} \leq n^{- 1/2 + O(c_0)}$$
\end{proposition}

\begin{proof}   The second claim follows from \cite[Proposition 66]{TVlocal1} and the third claim follows from \cite[Proposition 62]{TVlocal1}, so we turn to the first claim.

  
   The results in \cite{nguyen} only give a lower bound of $n^{-C}$ for some fixed $C$, which is not quite enough for our purposes.  On the other hand, if the atom distribution is sufficiently smooth, the claim follows from existing level repulsion estimates such as \cite{maltsev} or \cite{ESY3}, which are valid in the bulk region $|E| \leq 2-\delta$.  To extend to the case when the real and imaginary parts of the atom distribution are supported on at least three points, one can use the Four Moment Theorem (see \cite{TVlocal1}) and a moment matching argument (see e.g. the proof of \cite[Corollary 24]{TVlocal1}).  We remark that these are the only places in which we use the hypotheses that $|E| \leq 2-\delta$ and that the real and imaginary parts of the distribution are supported on at least three points.  It is likely that by improving the results in the above cited literature, one can remove these hypotheses\footnote{For instance, the results in \cite{vershynin} do not need the support hypothesis, but require the energy $E$ to be zero and the imaginary part to vanish.  It may however be possible to remove these hypotheses from the results in \cite{vershynin}, which could lead to an improvement of the proposition here.}.
\end{proof}

As a consequence of Proposition \ref{lsv-lsl}, we obtain a upper bound on the (imaginary part of the) Stieltjes transform:

\begin{corollary}\label{heft-cor}  For a sufficiently large constant $A_0 > 0$ (independent of $c_0$), one has
\begin{equation}\label{heft}
\Im s(E+\sqrt{-1} n^{-1-2A_0 c_0} ) \leq n^{-A_0 c_0}/2
\end{equation}
with high probability.
\end{corollary}

\begin{proof}  The left-hand side of \eqref{heft} can be written as
$$ n^{-2-2A_0 c_0} \sum_{i=1}^n \frac{1}{n^{-2-4A_0 c_0} + (\lambda_i(W_n)-E)^2}.$$
By Proposition \ref{lsv-lsl}, we assume with high probability that there are at most $O(n^{Ac_0})$ eigenvalues $\lambda_i(W_n)$ that are within $n^{-1+Ac_0}$ of $E$, but that all such eigenvalues are at least $n^{-1-c_0}$ away from $E$.  The total contribution of these eigenvalues to the above expression is then at most $O( n^{(-2A_0+A+2)c_0} )$.  Similarly, by using Proposition \ref{lsv-lsl} and dyadic decomposition of the spectrum around $E$, we see that the contribution of the eigenvalues that are further than $n^{-1+Ac_0}$ away are $O( n^{(-2A_0 - A) c_0} )$ with overwhelming probability.  Combining these bounds we obtain the claim if $A_0$ is large enough.
\end{proof}

Let $\chi: \R \to \R$ be a smooth cutoff to the region $|x| \leq n^{-A_0 c_0}$ that equals $1$ for $|x| \leq n^{-A_0 c_0}/2$.  From the above corollary,  $\chi( \Im s(E+\sqrt{-1} n^{-1-2A_0 c_0} ) )$ is equal to $1$ with high probability.  Thus it suffices to show that
\begin{equation}\label{eton}
\E G\left( n \Im \int_{\eta_0}^{n^{100}} s(E+\sqrt{-1}\eta)\ d\eta \right) \chi\left( \Im s(E+\sqrt{-1} n^{-1-2A_0 c_0} ) \right)
\end{equation}
is highly insensitive.

We now view $M'_n$ as obtained from $M_n$ by $n^2$ swapping operations, each of which either replaces a diagonal entry of $M_n$ with the corresponding entry of $M'_n$, or replaces the real or imaginary part of an off-diagonal entry of $M_n$ (and its adjoint) with the corresponding entries of $M'_n$ (leaving the other part of that entry unchanged).  Of these $n^2$ swapping operations, $n$ of them will involve a diagoanl entry, and the other $n^2-n$ 
 will involve an off-diagonal entry.  
It will suffice to show that each swapping operation only affects \eqref{eton} by $O(n^{-2-c})$ in the off-diagonal case and $O(n^{-1-c})$ in the diagonal case for some fixed $c>0$.  In fact we will obtain a bound of 
the form $O(n^{-5/2+O(c_0)})$ (where the implied constant may depend on $A, A_0$) in the off-diagonal case and $O(n^{-3/2+O(c_0)})$ in the diagonal case, which suffices for $c_0$ small enough.

Let $M^{(0)}_n$, $M^{(1)}_n$ be two adjacent matrices in this swapping process, and let $W^{(0)}_n, W^{(1)}_n$ be the associated normalised matrices.  Then we can write
\begin{align*}
W^{(0)}_n &= W_0 + \frac{1}{\sqrt{n}} \xi^{(0)} V\\
W^{(1)}_n &= W_0 + \frac{1}{\sqrt{n}} \xi^{(1)} V
\end{align*}
where $V$ is an elementary matrix, $\xi^{(0)}, \xi^{(1)}$ are real random variables matching to fourth order and bounded in magnitude by $O(n^{O(c_0)})$, and $W_0$ is a random matrix idnependent of both $\xi^{(0)}$ and $\xi^{(1)}$.  We can then write \eqref{eton} for $W^{(0)}_n$ using the notation of the preceding section as
$$
\E G( n \Im \int_{\eta_0}^{n^{100}} s_{\xi^{(0)}}(E+\sqrt{-1}\eta)\ d\eta ) \chi( \Im s_{\xi^{(0)}}(E+\sqrt{-1} n^{-1-2A_0 c_0}) ) 
$$
and we wish to show that this expression only changes by $O(n^{-5/2+O(c_0)})$ when $\xi^{(0)}$ is replaced by $\xi^{(1)}$ in the off-diagonal case, or $O(n^{-3/2+O(c_0)})$ in the diagonal case.

We now bound the resolvent:

\begin{lemma}[Resolvent bound]  If $\chi( s_{\xi^{(0)}}(E+i n^{-1-2A_0c_0} ) )$ is non-vanishing, then with overwhelming probability 
$$ \sup_{\eta > 0} \| R_{\xi^{(0)}}(E+\sqrt{-1}\eta) \|_{(\infty,1)} \ll n^{O(c_0)}$$
and
\begin{equation}\label{condition}
 \sup_{\eta > 0} \| R_0(E+i\eta) \|_{(\infty,1)} \ll n^{O(c_0)}
\end{equation}
\end{lemma}

\begin{proof}  From spectral decomposition one has
$$ \| R_\xi(E+\sqrt{-1}\eta) \|_{(\infty,1)}  \leq \sum_{i=1}^n \frac{\|u_j(W^{(0)}_n)\|_{\ell^\infty}^2}{|\lambda_i(W_n) - E - \sqrt{-1} \eta|}.$$

Applying the last statement\footnote{To be precise, we need to apply this statement for 
 $W^{(0)}_n$, but the proof for this matrix is the same as for $W_n$; see \cite{TVlocal1}.} in 
 Proposition \ref{lsv-lsl}, we conclude with overwhelming probability that
$$ \| R_\xi(E+\sqrt{-1}\eta) \|_{(\infty,1)}  \leq n^{-1+O(c_0)} \sum_{i=1}^n \frac{1}{|\lambda_i(W_n) - E|}.$$
Arguing as in Corollary \ref{heft-cor}, one see that if 
$\chi( s_\xi(E+\sqrt{-1} n^{-1-2A_0c_0} ) )$ is non-vanishing, then with overwhelming probability

$$ \sum_{i=1}^n \frac{1}{|\lambda_i(W_n) - E|} =O(n^{1+O(c_0)})$$
and the first claim follows.  The second claim then follows from Lemma \ref{neum} (swapping the roles of $0$ and $\xi^{(0)}$).
\end{proof}

We now condition to the event that \eqref{condition} holds.  To begin with, let us assume we are in the off-diagonal case.  Then by Proposition \eqref{taylor-prop} we have
$$ s_{\xi^{(i)}}(E+\sqrt{-1}\eta) = s_0(E+\sqrt{-1}\eta) + \sum_{j=1}^4 (\xi^{(i)})^j n^{-j/2} c_j(\eta) + O( n^{-5/2 + O(c_0)} )  \min \{1, \frac{1}{n\eta} \}$$
 for $i=0,1$,
where the coefficients $c_j$ enjoy the bounds
$$ c_j =O\left(n^{O(c_0)} \min\{ 1, \frac{1}{n\eta} \}\right).$$
From this and Taylor expansion above  we see that the expression
$$ G\left( n \Im \int_0^{n^{100}} s_{\xi^{(i)}}(E+\sqrt{-1}\eta)\ d\eta \right) \chi\left( \Im s_\xi(E+\sqrt{-1} n^{-1-2c_0} \right)$$
is equal to a polynomial of degree at most $4$ in $\xi^{(i)}$ with coefficients independent of $\xi^{(i)}$, plus an error of $O( n^{-5/2 + O(c_0)} )$. Taking the expectation and using the 
four moment assumption, we obtain that the difference between the expectations of 
$G$ with respect to $\xi^{(0)}$ and $ \xi^{(1)}$ is $O(n^{-5/2 +O(c_0)})$, as desired.  

In the diagonal case, one argues similarly, except that one only is assuming two matching moments, and so one should only Taylor expand to second order rather than fourth order.  This concludes the proof of Theorem \ref{four-moment}.

\appendix

\section{Moment calculations}\label{second}

In this appendix we establish Theorem \ref{secthm}.  Our main tool is the Leibniz expansion
$$ \det M_n = \sum_{\sigma \in S_n} I_\sigma,$$
where for each permutation $\sigma: \{1,\ldots,n\} \to \{1,\ldots,n\}$, $I_\sigma$ is the random variable
$$ I_\sigma := \sgn(\sigma) \prod_{i=1}^n \zeta_{i \sigma(i)}.$$

We begin with the first moment computation.  Clearly
$$ \E \det M_n = \sum_{\sigma \in S_n} \E I_\sigma.$$
Because all the $\zeta_{ij}$ have mean zero and are jointly independent on the upper triangular region $1 \leq i \leq j$, we see that $\E I_\sigma$ vanishes unless $\sigma$ consists entirely of $2$-cycles (i.e. is a perfect matching), in which case $\E I_\sigma = 1$.  Thus, $\E \det M_n$ is the number of perfect matchings on $\{1,\ldots,n\}$, which is easily seen to be zero when $n$ is odd and $\frac{n!}{(n/2)! 2^{n/2}}$ when $n$ is even.

Now we turn to the second moment computation for GOE, thus we seek the bounds
\begin{equation}\label{wig-2}
n^{3/2} n! \ll \E |\det M_n|^2 \ll n^{3/2} n!.
\end{equation}
We may of course assume that $n$ is large.

From the Leibniz expansion we have
\begin{equation}\label{wig}
\E |\det M_n|^2 = \sum_{\sigma,\rho \in S_n} \E I_\sigma \overline{I_\rho}.
\end{equation}
Actually, as GOE has real coefficients, we can omit the complex conjugate over the $I_\rho$ term.

Now we investigate the expressions $\E I_\sigma I_\rho$.  This expression can be estimated using the cycle decomposition of $\sigma$ and $\rho$.  It is not difficult to see that this expression will be zero unless the following conditions are satisfied:
\begin{itemize}
\item If $\gamma$ is a cycle in $\sigma$ of length other than two, then either $\gamma$ or its reversal $\gamma^{-1}$ is a cycle in $\rho$, and conversely.
\item The support of the $2$-cycles in $\sigma$ equals the support of the $2$-cycles in $\rho$.
\end{itemize}

Furthermore, if the above conditions are satisfied, then $\E I_\sigma \overline{I_\rho}$ is equal to $2^{C_1} 3^{c}$, where $C_1$ is the number of $1$-cycles of $\sigma$ (or of $\rho$), and $c$ is the number of $2$-cycles that are common to both $\sigma$ and $\rho$.  This comes from the fact that the diagonal entries of GOE have variance $2$, while the off-diagonal entries have a fourth moment of $3$.

Some elementary combinatorics shows that for a given permutation $\sigma$, the number of permutations $\rho$ obeying the above conditions is equal to
$$ \frac{2C_2!}{C_2! 2^{C_2}} \prod_{k \geq 3} 2^{C_k},$$
where $C_k$ is the number of $k$-cycles of $\sigma$. (To be more precise, we would have to write $C_k(\sigma)$, but as there is little chance of misunderstanding, we  prefer using just $C_k$ to simplify the presentation.)  Thus \eqref{wig} is thus lower bounded by
\begin{equation}\label{low}
 \sum_{\sigma \in S_n} 2^{C_1} \frac{2C_2!}{C_2! 2^{C_2}} \prod_{k \geq 3} 2^{C_k}.
\end{equation}
In the converse direction, for fixed $0 \leq c \leq C_2$, the number of $\rho$ obeying the above conditions and with exactly $c$ $2$-cycles in common is bounded by
$$ \binom{C_2}{c} \frac{2(C_2-c)!}{(C_2-c)! 2^{C_2-c}} \prod_{k \geq 3} 2^{C_k},$$
and so \eqref{wig} is upper bounded by
$$ \sum_{\sigma \in S_n} 2^{C_1} \sum_{c=0}^{C_2} 3^c \binom{C_2}{c} \frac{2(C_2-c)!}{(C_2-c)! 2^{C_2-c}} \prod_{k \geq 3} 2^{C_k}.$$
Let us first estimate the upper bound.  Observe from Stirling's formula that
$$ \frac{2(C_2-c)!}{(C_2-c)! 2^{C_2-c}} \ll (C_2-c)! 2^{C_2-c} / \sqrt{C_2-c+1}.$$
This and an elementary calculation show 
$$ \sum_{c=0}^{C_2} 3^c \binom{C_2}{c} \frac{2(C_2-c)!}{(C_2-c)! 2^{C_2-c}} \ll C_2! 2^{C_2} / \sqrt{C_2+1}.$$
So we may upper bound \eqref{wig} by
$$ O(\sum_{\sigma \in S_n} \frac{C_2! 2^{C_2}}{\sqrt{C_2+1}} \prod_{k \neq 2} 2^{C_k}).$$
Using the fact that $\sum_{m=1}^{n/2} \frac{1}{\sqrt m} = O (\sqrt n)$, we see 
that to prove the upper bound in \eqref{wig-2}, it  suffices to show that
\begin{equation}\label{c2m}
 \sum_{\sigma \in S_n: C_2 = m} m! 2^m \prod_{k \neq 2} 2^{C_k} \ll n \times n!
\end{equation}
for each $0 \leq m \leq n/2$. 

To establish \eqref{c2m}, we use a double counting argument\footnote{One could in fact obtain much more precise asymptotics on \eqref{c2m} using the method of generating functions, but we will not need to do so here.} 
 as follows.  For each permutation $\sigma$ with exactly $m$ $2$-cycles, we assign a quantity $F(\sigma)$ which is the product of 
 the number of ways to write down the $2$-cycles of $\sigma$ (counting ordering) and the number of ways to select some union $E$ 
 of  the $k$-cycles of $\sigma$ with $k \neq 2$.
 
 If the $2$-cycles of $\sigma$ are  $(x_1 y_1), \ldots ,(x_m y_m)$, then 
  there are $m! 2^m$ ways to write them down (counting all permutations in $S_m$ and the permutations between $x_j$ and $y_j$). 
  Furthermore, there are $\prod_{k \neq 2} 2^{C_k}$ ways to select $E$,  which is a $\sigma$-invariant set disjoint from the $x_1,\ldots,x_m,y_1,\ldots,y_m$.  This set $E$ has some cardinality $j$ between $0$ and $n-2m$. Therefore, 
 $$
 \sum_{\sigma \in S_n: C_2 = m} m! 2^m \prod_{k \neq 2} 2^{C_k} =  \sum_{\sigma \in S_n: C_2 = m}  F(\sigma) . $$

On the other hand, there are $\frac{n!}{(n-2m)!}$ ways to select  $2m$ 
ordered elements $x_1,\ldots,x_m,y_1,\ldots,y_m$ of $\{1,\ldots,n\}$.  For each $j$, there are then $\binom{n-2m}{j}$ ways to select $E$, and then to specify $\sigma$ on $E$ and on the complement of $E \cup \{x_1,\ldots,x_m,y_1,\ldots,y_m\}$ there are at most $j! (n-2m-j)!$ possibilities.  
(Notice that $\sigma$ restricted to $E$ is a permutation on $E$.) Putting all this together, we may bound the left-hand side of \eqref{c2m} by
$$ \sum_{j=0}^{n-2m} \frac{n!}{(n-2m)!} \binom{n-2m}{j} j! (n-2m-j)! = \sum_{j=0}^{n-2m} n!$$
and the claim follows.

Now we turn to the lower bound. From Stirling's formula we have
$$ \frac{2C_2!}{C_2! 2^{C_2}} \gg C_2! 2^{C_2} / \sqrt{C_2+1}$$
so by \eqref{low} (and crudely bounding $\sqrt{C_2+1}$ by $O(\sqrt{n})$ it suffices to show that
$$ \sum_{\sigma \in S_n} C_2! 2^{C_2} \prod_{k \neq 2} 2^{C_k} \gg n^2 \times n!.$$
For this, it suffices to prove the matching lower bound
\begin{equation}\label{c3m} \sum_{\sigma \in S_n: C_2 = m} m! 2^m \prod_{k \neq 2} 2^{C_k} \gg n \times n!
\end{equation}
to \eqref{c2m} for each $0 \leq m \leq n/4$ (say). 

We use the same double-counting argument as before.  We write the left-hand side of \eqref{c3m} as $\sum_{\sigma \in S_n: C_2=m} F(\sigma)$.  We use the  classical fact that as $n \to \infty$, the random variables $C_1,\ldots,C_k$ for any fixed $k$ converge jointly to independent Poisson variables of intensities $1/1, \ldots, 1/k$ respectively (see e.g. \cite{arratia}), so a positive constant fraction of $S_n$ is $2$-cycle free for $n \geq 0$. 
After fixing $x_1,\ldots,x_m,y_1,\ldots,y_m,E$, notice that  any $2$-cycle free permutation on $E$ and on its complement  will give a contribution to \eqref{c3m}. Thus, we obtain a lower bound of the form 
$$\gg \sum_{j=0}^{n-2m} \frac{n!}{(n-2m)!} \binom{n-2m}{j} j! (n-2m-j)! = n! (n-2m+1),$$
concluding the proof. 

Now we consider the second moment for the GUE case.  There are three differences here.  Firstly, the factor of $2^{C_1}$ that was present in the GOE analysis (which arose from the fact that the diagonal entries had variance $2$ instead of $1$) is now absent.  Secondly (and most importantly), in order for $\E I_\sigma \overline{I_\rho}$ to be non-vanishing, each cycle $\gamma$ of length at least three in $\sigma$ must appear in $\rho$ also; the appearance of the inverse cycle $\gamma^{-1}$ now leads to cancellation, in contrast with the GOE case.  As such, the $2^{C_k}$ factors for $k \geq 3$ are also absent.  Finally, the factor of $3^c$ in the above analysis becomes $2^c$, due to the smaller value of the fourth moment $\E |\zeta_{ij}|^4$ of the off-diagonal entries in the GUE case.  Repeating the GOE arguments, one reduces to showing that
$$
\sum_{\sigma \in S_n: C_2 = m} m! 2^m \ll n!
$$
for all $0 \leq m \leq n/2$, and
$$
\sum_{\sigma \in S_n: C_2 = m} m! 2^m \gg n!
$$
for all $0 \leq m \leq n/4$.  But this can be achieved by a routine modification of the above arguments (with the role of the additional set $E$, which represented the $\prod_{k \neq 2} 2^{C_k}$ factor, now omitted).

\begin{remark}  An inspection of the above argument shows that the hypotheses that $M_n$ are distributed according to GOE or GUE can be relaxed to the assertion that $M_n$ matches GOE or GUE to fourth order off the diagonal and to second order on the diagonal.
\end{remark}

\end{document}